\documentclass[11pt,leqno]{amsart}
\usepackage{amsmath,amssymb,amsthm}
\newtheorem{theorem}{Theorem}[section]
\newtheorem{lemma}[theorem]{Lemma}

\newtheorem{proposition}[theorem]{Proposition}
\newtheorem{corollary}[theorem]{Corollary}
\theoremstyle{definition}

\theoremstyle{remark}
\newtheorem{remark}[theorem]{Remark}
\numberwithin{equation}{section}

\def\fnote#1{\footnote}

\def\ignora#1{}
\def\n3#1{\left\vert  \! \left\vert \! \left\vert \, #1 \, \right\vert \!
  \right\vert \! \right\vert }

\begin{document}

\title{ Diameter two properties in James  spaces }

\author{Julio Becerra Guerrero, Gin{\'e}s L{\'o}pez-P{\'e}rez and Abraham Rueda Zoca}
\address{Universidad de Granada, Facultad de Ciencias.
Departamento de An\'{a}lisis Matem\'{a}tico, 18071-Granada
(Spain)} \email{glopezp@ugr.es, juliobg@ugr.es,
arz0001@correo.ugr.es}

\begin{abstract}
We study the diameter two properties in the spaces  $JH$,
$JT_\infty$ and $JH_\infty$. We show that the topological dual
space of the previous Banach spaces fails every diameter two
property. However, we prove that $JH$ and $JH_{\infty}$ satisfy
the strong diameter two property, and so the dual norm of these
spaces is octahedral. Also we find a closed hyperplane $M$ of
$JH_\infty$ whose topological dual space enjoys 
 the $w^*$-strong diameter two property and also $M$ and $M^*$ have an octahedral norm.
\end{abstract}

\maketitle\markboth{J. Becerra, G. L\'{o}pez and A. Rueda}{Diameter two properties in James spaces}

\section{Introduction.}
\bigskip
\par

The well known Radon-Nikodym property (RNP)
in Banach spaces is characterized by the existence of slices with
arbitrarily
 small diameter in every non empty, closed and bounded subset of the space. In last years a new topic has emerged,
 extremely different to the RNP: the diameter two properties.

We say that a Banach space has the slice diameter two property  (slice-D2P), respectively diameter two property (D2P),
strong diameter two property (SD2P) if every slice (respectively non-empty weakly open set, convex combination of slices) in its unit
 ball has diameter two. For a dual Banach space, we say that it has the $w^*$-slice diameter two property
 (respectively $w^*$-diameter two property, $w^*$-strong diameter two property) if every $w^*$ slice
 (respectively non-empty $w^*$ open set, convex combination of $w^*$-slices) in its unit ball has diameter two. It is known that the six above  properties are extremly different as proved in  \cite{avd}.

It is known a wide class of Banach spaces  enjoying some of the previous properties
as infinite-dimensional uniform algebras \cite{nywe}, infinite-dimensional $C^*$-algebras \cite{beloro},
non-reflexive $M$-embedded spaces \cite{gines}, Banach spaces with the Daugavet property \cite{sh}, etc, which shows that  these properties
have strong links with another well known  properties as  
containing  an isomorphic copy of $\ell_1$ \cite{blr} or Daugavet property.  

In spite of the wide study about the size of slices, non-empty
weakly open subsets and convex combination of slices in the unit
ball of several Banach spaces, non-classical Banach spaces
have not been deeply studied yet.

Probably the origin of non-classical Banach spaces was in
\cite{james1}, where the space $J$ (James space) is constructed in
order to provide an example of a non-reflexive Banach space which
fails to contain an isomorphic copy of $c_0$ or $\ell_1$. A year
later \cite{james2}, James went further and modified the
definition of the norm in order to show that $J$ and $J^{**}$ are
isometrically isomorphic, in spite of the non-reflexivity of $J$
(see \cite{fega} or \cite{litza}  for background about
$J$ space). It is known that $J$ and $J^*$ have the RNP as dual
and separable Banach spaces.

After the construction of $J$ space, James constructed in
\cite{james3} the $JT$ space (James tree space), exhibiting an
example of separable Banach space whose topological dual space is
not separable and that does not contain any isomorphic copy of
$\ell_1$, giving a negative answer to a conjecture of Stephan
Banach (again we refer to  \cite{fega} or \cite{litza}
for background in $JT$ space). It is known that $JT$ satisfies the
RNP and that $B$, the predual of $JT$, does not have the RNP
\cite{fega}. The fact that $B$ fails the RNP is far to prove that
$B$ has the slice-D2P. Indeed, in \cite[Theorem 5.1]{scsewe} it is
proved the existence of a constant $0<\beta<2$ such that every
closed and convex subset of the unit ball of $B$ has a slice of
diameter less than or equal to $\beta$, being the first
non-classical Banach space whose size of the slices of its unit
ball is studied. In fact, it is conjectured in \cite[Remark 5.2]{scsewe} that the above constant $\beta$ 
could be, at most, $\sqrt{2}$.

Motivated by the analysis of $B$, the aim of this note is to study
the slices of the unit ball of another exotic spaces. Indeed, in
section 2 we focus our attention in $JT_\infty$ space, showing
 that $JT_\infty^*$ fails the $w^*$-slice diameter two
property, and so every diameter two property. Note that
$JT_{\infty}$ is a separable dual space and then satisfies RNP. Then the space $B_{\infty}$,
the predual space of $JT_{\infty}$, fails every diameter two property. In fact, we prove that the inf of the diameters of 
slices in the unit ball of $B_{\infty}$ is, at most, $\sqrt{2}$. The same fact also holds for the predual space $B$ of $JT$, which proves that the suspect in \cite[Remark 5.2]{scsewe} holds for the unit ball.  In
section 3 we prove that the unit ball of $JH$ has a Fr\'echet
differentiability point and, as a consequence, the unit ball of
$JH^*$ contains $w^*$-slices of arbitrarily small diameter,
failing each property of diameter two. However it is proved in
this section that $JH$ in fact has the strong diameter two
property, and so $JH$ satisfies every diameter two property. As a
consequence the norm in the dual space $JH^*$ is octahedral.  In
section 4 we introduce the $JH_\infty$ space, a Banach space not isomorphic to $JH$, showing that unit
ball of $JH_\infty^*$ has $w^*$ slices of diameter strictly less
than 2. Also one can prove that  $JH_{\infty}$ has the strong diameter two property and so an octahedral dual norm.
Finally, in section 5, we find a closed hyperplane  $M$ of
$JH_\infty$ such that $M^*$ satisfies   the $w^*$-strong
diameter two property and also $M$ and $M^*$ have an octahedral norm.

We pass now to introduce some notation. We consider real Banach
spaces. $B_X$ and $S_X$  stand for the closed unit ball and the
unit sphere of the Banach space $X$. We denote by $X^*$ the
topological dual space of $X$. For a slice of a bounded subset
$A\subseteq X$ we mean the set

$$S(A,x^*,\alpha):=\{x\in A\ /\ x^*(x)>\sup x(A)-\alpha\}$$
for $x^*\in X^*$ and $0<\alpha$. By a $w^*$-slice of a
bounded subset $B\subseteq X^*$ we mean the set

$$S(B,x,\alpha):=\{x^*\in B\ /\ x^*(x)>\sup x(B)-\alpha\}$$
for $x\in X$ and $0<\alpha$.

Recall that the norm of a Banach space $X$ is octahedral (see \cite{dgz})  if for every $\varepsilon>0$ and for every finite-dimensional subspace $Y$ of $X$ there is $x\in S_X$ such that
$$\Vert \lambda x+y\Vert>(1-\varepsilon)(\vert \lambda\vert +\Vert y\Vert)$$
for every $y\in Y$ and for every scalar $\lambda$. We remark that the norm of a Banach space $X$ is octahedral if, and only if, $X^*$ satisfies the $w^*$-strong diameter two property and, dually,
the norm of $X^*$ is octahedral if, and only if, $X$ satisfies the strong diameter two property (see \cite{blr}).

Also we recall that a Banach space $X$ has the Daugavet property if the equation
\begin{equation}\label{daugavet}
 \Vert T+I\Vert =1+\Vert T\Vert
\end{equation} for every  rank one, linear and bounded operator on $X$, where $I$ denotes the identity operator. $X$ is said to have the almost Daugavet property if there is some norming subspace $Y$ of $X^*$ such that the equation \ref{daugavet} holds for every rank one operator $T$ given by $T=x\otimes y^*$ for $x\in X$ and $y^*\in Y$.  It is known \cite{Vladimir} that, for a separable Banach space, having octahedral norm and satisfying the almost Daugavet property are equivalent. Also  this is equivalent to say that $X^*$ has the $w^*$-strong diameter two property, as can be deduced from the comments in the above paragraph.   These facts will be used freely bellow.

The following known result, see Lemma 2.1 and Proposition 3.1 in \cite{bero},   will be useful in order to estimate the inf of the diameters of $w^*$- slices in dual spaces.

\begin{theorem}\label{teofundajulang}\
Let $X$ be a Banach space and assume that  $A\subseteq X^*$
satisfies  that $B_{X^*}=\overline{co}^{w^*}(A)$. If   $x\in S_X$,
then

$$\inf\limits_{\alpha>0} diam(S(A,x,\alpha))=\inf\limits_{\alpha>0} diam(S(B_{X^*},x,\alpha)).$$

\end{theorem}

\section{The space $JT_\infty$.}
\par
\bigskip

We shall begin with the  construction of $JT_\infty$ space. Define

$$T:=\{(\alpha_1,\ldots, \alpha_k)\ /\ k\in\mathbb N, \alpha_1,\ldots, \alpha_n\in\mathbb N\}\cup \{\emptyset\}.$$
Given $(\alpha_1,\ldots, \alpha_k),(\beta_1,\ldots, \beta_p)\in T$
we say that

$$(\alpha_1,\ldots, \alpha_k)\leq (\beta_1,\ldots, \beta_p)\Leftrightarrow \left\{\begin{array}{cc}
k\leq p & \ \\
\alpha_i=\beta_i & \forall i\in\{1,\ldots,k\}
\end{array}\right. .$$
Last relation defines a partial order defining $\vert
(\alpha_1,\ldots, \alpha_n)\vert=n$ and $\vert \emptyset\vert=0$.

By a segment we mean a subset $S\subseteq T$ totally ordered and finite.

Given $x:T\longrightarrow \mathbb R$ we consider

$$\Vert x\Vert=\sup\left( \sum_{i=1}^n \left( \sum_{t\in S_i}x(t)\right)^2\right)^\frac{1}{2},$$
where the sup is taken over all families $\{S_1,\ldots,S_n\}$
of disjoint segments in $T$.

Then $JT_\infty$ is defined as the completion of  the space of
finitely nonzero functions defined on $T$ in the above norm. Given
$\alpha\in T$ define

$$e_\alpha(\beta):=\left\{\begin{array}{cc}
1 & \mbox{if } \beta=\alpha\\
0 & \mbox{otherwise}
\end{array} \right.$$
Then $\{e_\alpha\}_{\alpha\in T}$ defines a Schauder basis on
$JT_\infty$. Denote by $\{e_\alpha^*\}_{\alpha\in T}$ the
biorthogonal sequence and let $B_{\infty}:=\overline{span}\{e_t^*\ /\ t\in
T\}$.

The space $JT_\infty$ was  introduced in \cite{goma}, where it is
proved that $B_{\infty}$, being the predual of  $JT_\infty$,   fails the Radon-Nikodym property. Furthermore, every infinite subspace
of $JT_\infty$ contains an isomorphic copy of $\ell_2$ and so
$JT_\infty$ does not contain isomorphic copies of $\ell_1$.

Now we pass to talk about the  size of slices in
$B_{JT_\infty^*}$. As in \cite{scsewe}, we define a molecule as a
functional of the form

$$x^*:=\sum_{i=1}^n \lambda_i f_{S_i}$$
for $S_1,\ldots, S_n$ disjoint segments in $T$ and $\sum_{i=1}^n
\lambda_i^2\leq 1$, where

$$f_S(x)=\sum_{t\in S} x(t)$$
for $S\subseteq T$ a segment.

Denote by $M$ the set of molecules in $JT_\infty^*$ and note that
$M\subseteq B_{JT_\infty^*}$.

We shall begin with the following Lemma, which shows that $M$ is a norming subset of $B_{JT_\infty^*}$.

\begin{lemma}\label{normantejt}\
$M$ is a norming subset of $B_{JT_\infty^*}$. As a consequence

\begin{equation}\label{describoladualjt}
B_{JT_\infty^*}=\overline{co}^{w^*}(M).
\end{equation}

\end{lemma}

\begin{proof}
Let $x\in S_{JT_\infty}$ a finitely nonzero function  defined on
$T$. Pick an arbitrary $0<\varepsilon<1$ and  take $0<\delta<1$
such that $(1-\delta)^2>1-\varepsilon.$ By the definition of the
norm in $JT_\infty$ we deduce that there exist $S_1,\ldots, S_n$
disjoint segments in $T$ such that

$$\left(\sum_{i=1}^n f_{S_i}(x)^2\right)^\frac{1}{2}>1-\delta.$$
For every $i\in\{1,\ldots, n\}$ define $\lambda_i:=f_{S_i}(x)$ and
note that by the definition of the norm in $JT_\infty$  it is
clear that $\sum_{i=1}^n \lambda_i^2\leq 1$. Moreover, in view of
last inequality, we have

$$\sum_{i=1}^n \lambda_i f_{S_i}(x)=\sum_{i=1}^n f_{S_i}(x)^2>(1-\delta)^2>1-\varepsilon.$$
As a consequence we can find in $M$ elements whose evaluation in
$x$ is as close to $\Vert x\Vert$ as desired. Hence $M$ is a
norming subset of $B_{JH_\infty^*}$.

From a separation argument we get now that
$B_{JT_\infty^*}=\overline{co}^{w^*}(M)$.

\end{proof}

Using previous lemma, we will prove that there exist $w^*$-slices in $B_{JT_\infty^*}$ with diameter strictly less than 2.

\begin{theorem}\
There exists $x\in S_{JT_\infty}$ such that

$$\inf\limits_{\alpha>0}diam\ S(B_{JT_\infty^*},x,\alpha)\leq \sqrt{2}.$$

\end{theorem}

\begin{proof}
Let $0<\varepsilon<1/2$. Pick $0<\delta<\min\{\varepsilon,\ 2\varepsilon(1-\varepsilon)\}$ and $0<\alpha<1/2$ such that $(1-\alpha)^2>1-\delta$. Define

$$x:=(1-\varepsilon)e_\emptyset+\varepsilon e_{(1)}\in S_{JT_\infty}.$$

Consider $S:=S(M,x,\alpha)$. Take $\sum_{i=1}^n \lambda_if_{S_i},\sum_{j=1}^m \mu_j f_{T_j}\in S$.

In view of the form of $x$ we can ensure the existence of $i\in\{1,\ldots,n\}, j\in\{1,\ldots,m\}$ such that $\{\emptyset,(1)\}\subseteq S_i\cap  T_j$. Indeed, it is clear that $(\cup_{i=1}^n S_i)\cap\{\emptyset,(1)\}\neq \emptyset$, since $\sum_{i=1}^n \lambda_if_{S_i}\in S$. Now it is not possible that $(\cup_{i=1}^n S_i)\cap\{\emptyset,(1)\}=\{(\emptyset)\}$ nor $(\cup_{i=1}^n S_i)\cap\{\emptyset,(1)\}=\{(1)\}$, since $0<\varepsilon<1/2$, $0<\alpha<1/2$, $\sum_{i=1}^n \lambda_i^2\leq 1$ and $\sum_{i=1}^n \lambda_if_{S_i}\in S$. Finally,  it is not possible that there exist $i\neq j$ such that $\{(\emptyset)\}\in S_i$ and $\{(1)\}\in S_j$, since if this is the case, we have that $(1-\varepsilon)\lambda_i+\varepsilon \lambda_j>1-\alpha$. Hence
$$(1-\alpha)^2<((1-\varepsilon)^2+\varepsilon^2)(\lambda_i^2+\lambda_j^2)\leq  (1-\varepsilon)^2+\varepsilon^2$$
and thus, using the conditions on $\alpha$, $\delta$ and $\varepsilon$, we get
$$1-2(\varepsilon(1-\varepsilon))<1-\delta<(1-\varepsilon)^2+\varepsilon^2,$$ which is a contradiction. This proves the existence of $i$ such that $\{\emptyset,(1)\}\subseteq S_i$. The same argument proves the existence of $j$ such that $\{\emptyset,(1)\}\subseteq T_j$.
Of course,  we assume without loss of generality that $i=j=1$. Now

$$\sum_{i=1}^n \lambda_i f_{S_i}(x)=\lambda_1(1-\varepsilon
+\varepsilon)=\lambda_1>1-\alpha\Rightarrow \lambda_1^2>(1-\alpha)^2>1-\delta.$$

As $\sum_{i=1}^n \lambda_i^2\leq 1$ then $\sum_{i=2}^n \lambda_i^2<\delta$. By a similar argument $\mu_1^2>1-\delta$ and hence $\sum_{j=2}^m \mu_j^2<\delta$.

In order to estimate $\left\Vert\sum_{i=1}^n \lambda_i f_{S_i}-\sum_{j=1}^m\mu_j f_{T_j}\right\Vert$ pick $y\in S_{JT_\infty}$. Hence

$$\left\vert\left(\sum_{i=1}^n \lambda_i f_{S_i}-\sum_{j=1}^m\mu_j f_{T_j}\right)(y)\right\vert\leq \mathop{\underbrace{\vert\lambda_1 f_{S_1}(y)-\mu_1 f_{T_1}(y)\vert}}\limits_{\mbox{{\large{A}}}}
+$$ $$\mathop{\underbrace{\left\vert \sum_{i=2}^n \lambda_i f_{S_i}(y)-\sum_{j=2}^m \mu_j f_{T_j}(y)\right\vert}}
\limits_{\mbox{{\large{B}}}}.$$

We shall begin estimating $B$. In view of H\"older's inequality we have

$$B\leq \sum_{i=2}^n \vert \lambda_i\vert\vert f_{S_i}(y)\vert+\sum_{j=2}^m \vert\mu_j\vert\vert f_{T_j}(y)\vert\leq$$

$$\leq \left(\sum_{i=2}^n \lambda_i^2+\sum_{j=2}^m \mu_j^2 \right)^\frac{1}{2}\left(\sum_{i=2}^n f_{S_i}(y)^2+\sum_{j=2}^m f_{T_j}(y)^2\right)^\frac{1}{2}\leq (2\delta)^\frac{1}{2}2^\frac{1}{2}=2\sqrt{\delta}$$

because $\sum_{i=2}^n f_{S_i}(y)^2\leq \Vert y\Vert^2=1, \sum_{j=2}^m f_{T_j}(y)^2\leq 1$ due to the disjointness of  $\{S_2,\ldots,S_n\}$ and $\{T_2,\ldots,T_m\}$.
 So $B\leq 2\sqrt{\delta}$.  Now we will estimate $A$:

$$A\leq \vert \lambda_1-\mu_1\vert \vert f_{T_1\cap S_1}(y)\vert+\vert \lambda_1\vert \vert f_{S_1\setminus T_1}(y)\vert+\vert\mu_1\vert\vert  f_{T_1\setminus S_1}(y)\vert.$$

As $1\geq \lambda_1>1-\alpha$ and $1\geq \mu_1>1-\alpha$ then $\vert \lambda_1-\mu_1\vert<\alpha$. Hence

$$A\leq \alpha\Vert f_{T_1\cap S_1}\Vert \Vert y\Vert+\vert \lambda_1\vert \vert f_{S_1\setminus T_1}(y)\vert+\vert\mu_1\vert\vert  f_{T_1\setminus S_1}(y)\vert=$$

$$\alpha+\vert \lambda_1\vert \vert f_{T_1\setminus S_1}(y)\vert+\vert\mu_1\vert\vert  f_{T_1\setminus S_1}(y)\vert\leq \alpha+\vert f_{S_1\setminus T_1}(y)\vert+\vert f_{T_1\setminus S_1}(y)\vert.$$

Again applying H\"older's inequality we have

$$A\leq \alpha+\sqrt{2}\left(f_{S_1\setminus T_1}(y)^2+f_{T_1\setminus S_1}(y)^2\right)^\frac{1}{2}.$$

As $\{S_1\setminus T_1,T_1\setminus S_1\}\subseteq T$ is a family of disjoint segments we have that $f_{S_1\setminus T_1}(y)^2+f_{T_1\setminus S_1}(y)^2\leq \Vert x\Vert^2=1$. Hence

$$A\leq \alpha+\sqrt{2}.$$

Summarizing

$$\left\vert\left(\sum_{i=1}^n \lambda_i f_{S_i}-\sum_{j=1}^m\mu_j f_{T_j}\right)(y)\right\vert\leq \alpha+\sqrt{2}+2\sqrt{\delta}.$$

From the arbitrariness of $y\in S_{JT_\infty}$ we have that

$$\left\Vert\sum_{i=1}^n \lambda_i f_{S_i}-\sum_{j=1}^m \mu_j f_{T_j} \right\Vert=\sup\limits_{y\in S_{JT_\infty}}\left\vert\left(\sum_{i=1}^n \lambda_i f_{S_i}-\sum_{j=1}^m\mu_j f_{T_j}\right)(y)\right\vert\leq$$ $$\sqrt{2}+\alpha+2\sqrt{\delta}.$$

Hence

$$diam(S)\leq \sqrt{2}+\alpha+2\sqrt{\delta}.$$

So

$$\inf\limits_{\alpha>0} diam(S(M,x,\alpha))\leq \sqrt{2}+2\sqrt{\delta}.$$

Since $0<\delta<\varepsilon$ was arbitrary we deduce that

$$\inf\limits_{\alpha>0} diam(S(M,x,\alpha))\leq \sqrt{2}.$$

In view of Lemma \ref{normantejt}, Theorem \ref{teofundajulang} applies and

$$\inf\limits_{\alpha>0}diam(S(B_{JT_\infty^*},x,\alpha)),$$

so we are done.

\end{proof}

In view of previous theorem, for each $0<\varepsilon<2-\sqrt{2}$ we can find $S$ a $w^*$-slice in $B_{JT_\infty^*}$ such that $diam(S)<\sqrt{2}+\varepsilon$. In particular, $JT_\infty^*$ fails to have the $w^*$-slice diameter two property and hence $B_{\infty}$ fails every diameter two property, since the inf of the diameters of slices in the unit ball of $B_{\infty}$ agrees with the inf of the diameters of $w^*$-slices in the unit ball of $JT_{\infty}^*$. In fact, this inf is, at most, $\sqrt{2}$. Also, it is possible obtaining the same result for the space $B$, the predual of $JT$, with the above proof, which shows that the conjecture in \cite{scsewe} that the inf of diameters of slices in the unit ball in $B$ is, at most, $\sqrt{2}$  holds.

\section{The space $JH$.}
\par
\bigskip

We shall begin with the construction of $JH$ space. Following \cite{fega} we denote by

$$T:=\left\{(n,i)\ /\ 0\leq n<\infty, 0\leq i< 2^n\right\}$$
the diadic tree. We say that $(n+1,2i)$ and $(n+1,2i+1)$ are
offspring of $(n,i)$ for every $(n,i)\in T$. A segment will be a
non-empty finite sequence

$$S=\{t_1,\ldots, t_n\}$$
such that $t_{j+1}$ is an offspring of $t_j$ for every $j\in
\{1,\ldots, n-1\}$.

Now we are ready to define  a partial order in $T$: given
$t_1,t_2\in T$ we say that $t_1<t_2$ if, and only if, $t_1\neq
t_2$ and there exists a segment such that $t_1$ is the first
element of the segment and $t_2$ is the last element on it.

The set

$$\{(n,i)\ /\ 0\leq i<2^n\}$$
is called the $n$-th level of $T$ for every $0\leq n<\infty$.

Given $n,m\in\mathbb N, n\leq m$ we will say that a subset $S\subseteq T$ is an $n-m$ segment if

\begin{itemize}
\item For every $n\leq k\leq m$ there exists an only element in $S$ which is in the $k$-th level of $T$,
\item If $(p,i),(q,j)\in S$ and $p<q$ then $(p,i)<(q,j)$ (in other words, $S$ is a totally ordered subset of $T$).
\end{itemize}

Given $x:T\longrightarrow \mathbb R$ and  $S\subseteq T$  a segment in $T$,  we define

$$f_S(x):=\sum_{t\in S} x(t).$$
Note that the above sum is well defined because $S$ is finite.

Given $\{S_1,\ldots,S_n\}$ a family of segments in $T$ we say that are admissible if:

\begin{enumerate}
\item[i)] There exist $p\leq q$ natural numbers such that $S_i$ is
an $p-q$ segment for every $i\in\{1,\ldots, n\}$. \item[ii)]
$S_i\cap S_j=\emptyset$ whenever $i\neq j$.
\end{enumerate}

Given $x:T\longrightarrow \mathbb R$ a finitely nonzero function we define

$$\Vert x\Vert:=\max \sum_{i=1}^n\vert f_{S_i}(x)\vert=\max \sum_{i=1}^n \left\vert \sum_{t\in S_i} x(t) \right\vert,$$
where the maximum is taken over all families $S_1,\ldots, S_n$ of
admissible segments in $T$.

Now define $JH$ as the completion of the space of finitely nonzero functions on $T$ in the above norm.

Given $t\in T$ we define $e_t\in JH$ by the equation

$$e_t(s):=\left\{\begin{array}{cc}
1 & \mbox{if } t=s\\
0 & \mbox{otherwise}
\end{array} \right. .$$
Then $\{e_t\}_{t\in T}$ defines a Schauder basis on $JH$.

$JH$ space was introduced by J.Hagler in \cite{hag}. It  is proved
in that paper that $JH$ is a separable Banach space such that
$JH^*$ is not separable and that every infinite-dimensional
subspace of $JH$ contains an isomorphic copy of $c_0$. In
particular $JH$ contains an isomorphic copy of $c_0$, so it can
not be a dual space \cite[Proposition 2.e.8]{litza}.

\begin{lemma}\label{lemanormaelenuevo}

Let $x:T\longrightarrow \mathbb R$ a finitely non-zero function and $n\in\mathbb N\setminus\{1\}$ such that

$$\Vert x\Vert\leq 1-\frac{1}{n}.$$

Pick $a\in T$ such that $lev(a)>\max\limits_{t\in supp(x)} lev(t)$. Let $\ell\in\mathbb N$ big enough such that there exists $t_1,\ldots, t_n\in T$ such that

\begin{itemize}
\item $lev(t_i)=\ell$ for each $i$.
\item $a<t_i$ for all $i$.
\end{itemize}

If we define $y:T\longrightarrow \mathbb R$ such that

$$y(t):=\left\{\begin{array}{cc}
x(t) & t\in supp(x)\\
\mu_i\frac{1}{n} & t=t_i\ i\in\{1,\ldots, n\}, \mu_i\in\{-1,1\} \\
0 & \mbox{otherwise}
\end{array}\right . ,$$

then $\Vert y\Vert\leq 1$.

\end{lemma}

\begin{proof}

Let $\{S_1,\ldots, S_k\}$ be a family of admissible segments in $T$, $\lambda_1,\ldots, \lambda_k\in\{-1,1\}$ and define

$$x^*:=\sum_{i=1}^k \lambda_i f_{S_i}.$$

In order to prove that $\Vert y\Vert\leq 1$ we have to prove that $x^*(y)\leq 1$, following the definition of the norm in $JH$.  

By cases:

\begin{enumerate}
\item $\bigcup\limits_{i=1}^k S_i\cap \{t_1,\ldots, t_n\}=\emptyset$.

In this case we have, in view of the definition of $y$ that

$$x^*(y)=x^*(x)\leq \Vert x\Vert\leq 1-\frac{1}{n}$$

by hypothesis.
\item $\bigcup\limits_{i=1}^k S_i\cap \{t_1,\ldots, t_n\}\neq\emptyset$ but $\bigcup\limits_{i=1}^k S_i\cap supp(x)=\emptyset$.

In this case we have

$$x^*(y)=\sum_{i=1}^k \lambda_i f_{S_i}(y)\leq \sum_{i=1}^n \frac{1}{n}=1$$

\item  $\bigcup\limits_{i=1}^k S_i\cap \{t_1,\ldots, t_n\}\neq\emptyset$ and $\bigcup\limits_{i=1}^k S_i\cap supp(x)\neq\emptyset$.

Finally, in this case we have that there exists one only $i\in\{1,\ldots, n\}$ such that $a\in S_i$ (otherwise $\bigcup\limits_{i=1}^k S_i\cap \{t_1,\ldots, t_n\}=\emptyset$ in view of the order defined on $T$).
We can assume, without of generality, that $i=1$. If $S_j$ is a $p-q$ segment, we can write

$$S_j:=T_j\cup R_j$$

where $T_j$ is a $p-(\ell-1)$ segment and $R_j$ is a $\ell-q$ segment for each $j\in\{1,\ldots, k\}$.

In view of the disjointness of $S_1,\ldots, S_k$ we have  for each $j\in\{2,\ldots, n\}$ that $S_j\cap \{t_1,\ldots, t_n\}=$ $\emptyset$. In addition, as $\ell>\max\limits_{t\in supp(x)}lev(t)$ we deduce that

$$f_{R_i}(y)=0\ \forall i\in\{2,\ldots, n\}.$$

Hence

$$x^*(y)=\sum_{i=1}^k \lambda_i f_{T_i}(y)+\lambda_1 f_{R_1}(y).$$

Now we have that $\{T_1,\ldots, T_k\}$ is a family of admissible segments on $T$. Hence

$$x^*(y)\leq \Vert x\Vert+\lambda_1 f_{R_1}(x)\leq 1-\frac{1}{n}+f_{R_1}(y).$$

Now, as $\{t_1,\ldots, t_n\}$ are incomparable nodes on $T$ at the same level we have that $\{t_1,\ldots, t_n\}\cap R_1$ has one element. Hence

$$x^*(y)\leq  1-\frac{1}{n}+f_{R_1}(y)\leq1-\frac{1}{n}+\frac{1}{n}=1.$$

\end{enumerate}

By the previous discussion we deduce that $\Vert y\Vert\leq 1$ as desired.

\end{proof}

\begin{theorem}\label{sliced2pjh}

$JH$ has the strong diameter two property (and so the norm of $JH^*$ is octahedral).

\end{theorem}

\begin{proof}

Let $C:=\sum_{i=1}^n \lambda_i S(B_{JH},x_i^*,\alpha)$ a convex combination of slices on $B_{JH}$. Let prove that $diam(C)=2$.

To this aim pick $x_i:T\longrightarrow \mathbb R$ a finitely non-zero supported function on $T$ such that $\Vert x_i\Vert<1$ and

$$x_i^*(x_i)>1-\alpha,$$

for each $i\in\{1,\ldots, n\}$. For each $i\in\{1,\ldots, n\}$ we can find $a_i\in supp(x_i)$ such that $lev(a_i)=\max\limits_{t\in supp(x_i)} lev(t)$.

As $\Vert x_i\Vert<1$ for each $i\in\{1,\ldots, n\}$ we can find $m\in\mathbb N$ such that $\Vert x_i\Vert\leq 1-\frac{1}{m}$ for each $i\in\{1,\ldots, n\}$. Now we can find $a\in T$ such that $lev(a)>\max\limits_{1\leq i\leq n} lev(a_i)$, $k>\max\limits_{1\leq i\leq n} lev(a_i)$ big enough and $\{t_1^i,\ldots, t_{2n}^i\}$ a family of nodes on $T$ at level $k$ such that $a<t_p^i$ for each $i\in\{1,\ldots, n\}, p\in\{1,\ldots, 2m\}$ and that

$$t_p^i\neq t_q^j\ \ \mbox{if }\ i\neq j\ \mbox{ or }\ p\neq q.$$

In other words, last condition guaranties that $\{t_p^i\ /\ i\in\{1,\ldots, n\}, j\in\{1,\ldots, 2m\}\}$ is a family of nodes pairwise different nodes at level $k$ which are bigger than $a$.

For each $i\in\{1,\ldots, n\}$ we define $y_i,z_i:T\longrightarrow \mathbb R$ a finitely non-zero function on $T$ as follows

$$y_i(t):=\left\{\begin{array}{ccc}
x_i(t) & \mbox{if} & t\in supp(x_i)\\
sign\left(x_i^*\left(e_{t_p^i}
\right)\right)\frac{1}{m} & t=t_p^i & p\in\{1,\ldots, m\}\\
0 & \ & \mbox{otherwise}
\end{array} \right.$$

and

$$z_i(t):=\left\{\begin{array}{ccc}
x_i(t) & \mbox{if} & t\in supp(x_i)\\
sign\left(x_i^*\left(e_{t_p^i}
\right)\right)\frac{1}{m} & t=t_p^i & p\in\{m+1,\ldots, 2m\}\\
0 & \ & \mbox{otherwise}
\end{array} \right. .$$

In view of Lemma \ref{lemanormaelenuevo} we have that $\Vert y_i\Vert\leq 1$ and $ \Vert z_i\Vert\leq 1$.

Let prove that, in fact, $y_i,z_i\in S(B_{JH},x_i^*,\alpha)$ for each $i\in\{1,\ldots, n\}$. To this aim pick $i\in\{1,\ldots, n\}$ and we shall prove that $y_i\in S(B_{JH},x_i^*,\alpha)$, being the case of $z_i$ similar. Using the linearity of $x_i^*$ we have

$$x_i^*(y_i)=x_i^*(x_i)+\sum_{p=1}^m\frac{1}{m}sign\left(x_i^*\left(e_{t_p^i}
\right)\right)x_i^*\left(e_{t_p^i}\right)=$$

$$=x_i^*(x_i)+\sum_{p=1}^m\frac{1}{m}\left\vert x_i^*\left(e_{t_p^i}\right)\right\vert\geq x_i^*(x_i)>1-\alpha.$$

Hence $\sum_{i=1}^n \lambda_i y_i,\sum_{i=1}^n \lambda_i z_i\in C$. Then

$$diam(C)\geq \left\Vert \sum_{i=1}^n \lambda_i y_i-\sum_{i=1}^n \lambda_i z_i\right\Vert.$$

Now we shall prove that $\left\Vert \sum_{i=1}^n \lambda_i y_i-\sum_{i=1}^n \lambda_i z_i\right\Vert=2$. To this aim check that $\{ \{t_p^i\}\ /\ i\in\{1,\ldots, n\}, p\in\{1,\ldots, 2m\}\}$ is a family of admissible segments on $T$. Hence

$$f:=\sum_{i=1}^n \sum_{p=1}^m sign\left(x_i^*\left(e_{t_p^i}
\right)\right)f_{\{t_p^i\}}-\sum_{p=m+1}^{2m} sign\left(x_i^*\left(e_{t_p^i}
\right)\right)f_{\{t_p^i\}}$$

is an element on $JH^*$ whose norm is less or equal to one (in view of the definition of the norm in $JH$). So

$$\left\Vert \sum_{i=1}^n \lambda_i y_i-\sum_{i=1}^n \lambda_i z_i\right\Vert\geq f\left(\sum_{i=1}^n \lambda_i y_i-\sum_{i=1}^n \lambda_i z_i\right)=$$

$$=\sum_{i=1}^n \lambda_i \frac{1}{m} \sum_{p=1}^m  sign\left(x_i^*\left(e_{t_p^i}
\right)\right)^2+\lambda_i\frac{1}{m} \sum_{p=m+1}^{2m}  sign\left(x_i^*\left(e_{t_p^i}
\right)\right)^2=$$

$$=2\sum_{i=1}^n \lambda_i =2.$$

So $\left\Vert \sum_{i=1}^n \lambda_i y_i-\sum_{i=1}^n \lambda_i z_i\right\Vert=2$, as wanted.

\end{proof}

We pass now to study the diameter two property on $JH^*$. Our aim
is to prove that $B_{JH^*}$ has $w^*$-slices with arbitrary small
diameter. In fact, we will find $x\in S_{JH}$ such that
$\inf\limits_{\alpha>0} diam(S(B_{JH^*},x,\alpha))=0$.

If we denote by

$$A:=\left\{\sum_{i=1}^n \lambda_i f_{S_i}\ \left/\begin{array}{c}
\lambda_i\in \{-1,1\}\\
\{S_1,\ldots, S_n\}\mbox{ is a family of admissible segments in }T
\end{array} \right. \right\},$$
it is clear that $A\subseteq B_{JH^*}$ is a norming subset (by the
definition of the norm on $JH$). Hence

$$\overline{co}^{w^*}(A)=B_{JH^*}$$
by Hahn-Banach theorem.

Now we are ready to show that $B_{JH^*}$ has $w^*$-slices of arbitrarily small diameter.

\begin{theorem}
There exists $x\in S_{JH}$ satisfying that

$$\inf\limits_{\alpha>0} diam(S(B_{JH^*},x,\alpha))=0.$$

\end{theorem}

\begin{proof}

Pick $0<\varepsilon<\frac{1}{4}$ and  let

$$x=(1-\varepsilon)e_{(0,0)}+\varepsilon e_{(1,0)}-\varepsilon e_{(1,1)}-\varepsilon e_{(2,0)}-\varepsilon e_{(2,1)}-\varepsilon e_{(2,2)}+\varepsilon e_{(2,3)}.$$
It is clear that $\Vert x\Vert\geq 1$ considering  the family of
admissible segments $\{\{(0,0),(1,0)\}\}$. It can also be checked
that if $\{S_1,\ldots, S_r\}$  is a family of admissible segments
in $T$ which is different of the family $\{\{(0,0),(1,0)\}\}$ then

$$\sum_{i=1}^r\left\vert\sum_{t\in S_i} x(t)\right\vert\leq \max\{1-\varepsilon,4\varepsilon\}<1.$$
Hence $\Vert x\Vert=1$. Moreover, if we take $\{S_1,\ldots,S_r\}$
a family of admissible segment, $\lambda_1,\ldots, \lambda_r\in
\{-1,1\}$ such that

$$\sum_{i=1}^r \lambda_i f_{S_i}(x)>1-\alpha$$
for $0<\alpha<\min\{1-4\varepsilon,\varepsilon\}<1$ then $r=1$,
$S_1=\{(0,0),(1,0)\}$ and $\lambda_1=1$. So

$$S(A,x,\alpha)=\left\{f_{\{(0,0),(1,0)\}}\right\}\Rightarrow \inf\limits_{\alpha>0}diam(S(A,x,\alpha))=0.$$
Now Theorem \ref{teofundajulang} applies and as a consequence we
get that

$$\inf\limits_{\alpha>0} diam(S(B_{JH^*},x,\alpha))=0,$$
so we are done.\end{proof}

Let's remark that $x$ of the above theorem is a Fr\'echet
differentiability point of $B_{JH}$, see \cite{dgz}, so as a
consequence of the above result we deduce that the unit ball $JH^*$ has denting points.

\section{The space $JH_\infty$.}

\par
\bigskip

We shall begin with the construction of $JH_\infty$ space from the $JH$ space, by a similar process to the construction of $JT_\infty$ space from $JT$ space.

Consider $T$ as in section 2. A segment $S=\{t_1,\ldots, t_k\}$ is a $n-m$ segment, for $n\leq m$, if $\vert t_1\vert=n$ and $\vert t_k\vert=m$.

If $\{S_1,\ldots, S_k\}$ is a finite family of segments in $T$ we say that is admissible if

\begin{enumerate}
\item Exist natural numbers $n,m$ satisfying $n\leq m$ and $S_i$ is a $n-m$ segment for every $i\in\{1,\ldots, k\}$.
\item $S_i\cap S_j=\emptyset$ if $i\neq j$.
\end{enumerate}

Given $x:T\longrightarrow \mathbb R$ a finitely nonzero function we define

$$\Vert x\Vert:=\sup\sum_{i=1}^k \left\vert \sum_{t\in S_i}x(t)\right\vert,$$

where the sup is taken over all families of admissible segments $\{S_1,\ldots,S_k\}$ in $T$.

We define the space $JH_\infty$ as the completion of the space of finitely nonzero functions on $T$ in the above norm.

If $S\subseteq T$ is a segment then we denote by

$$f_S(x):=\sum_{t\in S} x(t).$$

Note that $f_S\in S_{(JH_\infty)^*}$.

Moreover, in view of the definition of the norm we have that given   a family of admissible segments $\{S_1,\ldots, S_k\}$ then $\sum_{i=1}^k \lambda_i f_{S_i}\in B_{(JH_\infty)^*}$, whenever $\lambda_1,\ldots, \lambda_k\in \{-1,1\}$.

Given $\alpha\in T$ define

$$e_\alpha(\beta):=\left\{\begin{array}{cc}
1 & \mbox{if } \beta=\alpha\\
0 & \mbox{otherwise}
\end{array} \right.$$

Then $\{e_\alpha\}_{\alpha\in T}$ defines a Schauder basis on $JH_\infty$.

Let us remark that $JH_{\infty}$ is not isomorphic to $JH$. Indeed, we know that $JH$ does not contain isomorphic copies of $\ell_1$, however it is enough consider the sequence $\{e_{\alpha_n}\}$, where 
$\{\alpha_n\}$ is an infinite sequence of immediate successors of the first node  in $T$, to get an isometric copy of the usual basis in $\ell_1$. Furthermore it is clear that $JH_{\infty}$ contains isometric copies of $JH$. Now, we can get, as in the previous section for $JH$, the following result.

\begin{theorem} $JH_{\infty}$ has the strong diameter two property (and so the norm of $JH_{\infty}^*$ is octahedral).
\end{theorem}

In order to study diameter two properties in $JH_{\infty}^*$,  next Lemma will help us  to estimate the diameter of certain  $w^*$-slices in $B_{JH_\infty^*}$.

\begin{proposition}\label{estimajhinf}
Let $R,S$ be two disjoint segments in $T$ such that are $p-q$ and $p-r$ segments for suitable $p,q,r\in\mathbb N, p\leq q\leq r$. Then

$$\Vert f_R-f_S\Vert\leq \frac{5}{3}.$$

\end{proposition}

\begin{proof}
If $r=q$ then $\{S,R\}$ is a family of admissible segments in $T$. Hence

$$\Vert f_R-f_S\Vert=1<\frac{5}{3}.$$

Now assume that $q<r$. Then we can find $U$ a $p-q$ segment and $V$ an $(q+1)-r$ segment such that

\begin{equation}\label{descosegran}
U\cup V=R\Rightarrow f_R=f_U+f_V.
\end{equation}

Let $\alpha\in\mathbb R^+_0$ such that $\Vert f_R-f_S\Vert=2-\alpha$ and $\varepsilon\in\mathbb R^+$. Then there exists a finitely nonzero function $x:T\longrightarrow \mathbb R$, $\Vert x\Vert\leq 1$, such that

$$(f_R-f_S)(x)>2-\alpha-\varepsilon\Rightarrow f_R(x)>1-\alpha-\varepsilon\ \mbox{ and } f_S(x)<-1+\alpha+\varepsilon.$$

As $U$ is a $p-q$ segment disjoint with $S$ we have that $\{U,S\}$ is a family of admissible segments. As a consequence $\Vert f_U-f_S\Vert\leq 1$. Hence

$$2-\alpha-\varepsilon<f_R(x)-f_S(x)=(f_U-f_S)(x)+f_V(x)\leq 1+f_V(x)$$ and so  $$f_V(x)>1-\alpha-\varepsilon.$$

Moreover

$$1\geq f_R(x)=f_U(x)+f_V(x)\geq 1-\alpha-\varepsilon+f_U(x),$$

hence

\begin{equation}\label{estimasegaltoar}
f_U(x)\leq \alpha+\varepsilon.
\end{equation}

Now, again using that $\{S,U\}$ is a family of admissible segments, we have that $\Vert f_U+f_S\Vert\leq 1$. Hence

$$-1\leq (f_U+f_S)(x)< f_U(x)+(-1+\alpha+\varepsilon).$$

Then

\begin{equation}\label{estimasegaltoaba}
f_U(x)>-\alpha-\varepsilon.
\end{equation}

Combining both (\ref{estimasegaltoar}) and (\ref{estimasegaltoaba}) it follows

\begin{equation}\label{estimamodulo}
\vert f_U(x)\vert\leq \alpha+\varepsilon.
\end{equation}

Now, as $x$ has finite support, we can find $W$ a $(q+1)-r$ segment such that $S\cup W$ is a $p-r$ segment disjoint with $R$ and we can assume that $x(t)=0\ \forall t\in W$. From here, we deduce that $\{R,S\cup W\}$ is a family of admissible segments in $T$. Hence

$$1\geq \Vert x\Vert\geq \vert f_R(x)\vert+\vert f_{S\cup W}(x)\vert=\vert f_U(x)+f_V(x)\vert+\vert f_S(x)\vert\geq$$

$$\geq \vert f_V(x)\vert-\vert f_U(x)\vert+\vert f_S(x)\vert\geq (1-\alpha-\varepsilon)-(\alpha+
\varepsilon)+(1-\alpha-\varepsilon)=
2-3\alpha-3\varepsilon.$$

From the arbitrariness of $\varepsilon$ we deduce that

$$1\geq 2-3\alpha\Rightarrow \alpha\geq \frac{1}{3}.$$

Then $\Vert f_S-f_T\Vert=2-\alpha\leq 2-\frac{1}{3}=\frac{5}{3}$, as wanted.

\end{proof}

Now we can conclude that there are $w^*$-slices in $B_{JH_\infty^*}$ with diameter strictly less than two. In fact, we can find $w^*$-slices with diameter less than $\frac{5}{3}+\varepsilon$ for every $0<\varepsilon<\frac{1}{3}$.

\begin{theorem}\label{nowsliced2pjhinf}\
There exists $x\in S_{JH_\infty}$ such that

$$\inf\limits_{\alpha>0} S(B_{JH_\infty^*},x,\alpha)\leq\frac{5}{3}.$$

\end{theorem}

\begin{proof}
Define

$$A:=\left\{ \sum_{i=1}^n\lambda_i f_{S_i}\ \left/\begin{array}{c}
\vert \lambda_i\vert= 1 \ i\in\{1,\ldots, n\}\\
 \{S_1,\ldots, S_n\}\mbox{ family of admissible segments}\end{array} \right.\right\}.$$

It is clear that $\overline{co}^{w^*}(A)=B_{JH_\infty^*}$ by an easy separation argument. Let $0<\delta<1$.

Define $x:=(1-\delta)e_{\emptyset}+\delta e_{(1)}$. Pick $0<\alpha<\delta< 1/2$. Then if $\sum_{i=1}^n \lambda_i f_{S_i}\in S(A,x,\alpha)$ we have that $n=1$, $S_1$ is a $0-p$ segment for suitable $p\geq 1$, $\emptyset,(0)\in S_1$ and $\lambda_1=1$.

So, in order to estimate $diam(S(A,x,\alpha))$ pick $f_S,f_R\in S(A,x,\alpha)$. Notice that $S\cap R\neq \emptyset$ (both segments contain the set $\{\emptyset,(1)\}$). However we can find $U,V$ two disjoint segments which are $p-q$ and $p-r$ segments, for suitable $p,q,r\geq 2$ such that

$$S=(S\cap R)\cup U\ \mbox{ and }\ R=(S\cap R)\cup V.$$

Then

$$f_R-f_S=f_{S\cap R}+f_V-f_{S\cap R}-f_U=f_V-f_U.$$

By Proposition \ref{estimajhinf} we deduce that

$$\Vert f_V-f_U\Vert\leq \frac{5}{3}\Rightarrow \Vert f_R-f_S\Vert\leq \frac{5}{3}.$$

From the arbitrariness of $f_R,f_S\in S(A,x,\alpha)$ we deduce that

$$diam(S(A,x,\alpha))\leq \frac{5}{3}.$$

Hence

$$\inf\limits_{\alpha>0}diam(S(A,x,\alpha))\leq \frac{5}{3}.$$

Now theorem \ref{teofundajulang} applies and

$$\inf\limits_{\alpha>0} diam(S(B_{JH_\infty^*},x,\alpha))\leq \frac{5}{3}$$

so we are done.

\end{proof}

In particular, the above theorem shows that $JH_\infty^*$ fails to have the $w^*$-slice diameter two property, and so every diameter two property.

In view of the element $x$ in last theorem, it seems that the fact that $\emptyset\in supp(x)$ is a very important fact (it allowed us to describe easily the elements of $S(A,x,\alpha)$). This fact will become clear in next section.

\section{An hyperplane of $JH_{\infty}^*$ satisfying the $w^*$-strong diameter two property.}

We will consider $T$ defined as in the previous section. Let

$$N:=\left\{x:T\longrightarrow
\mathbb R\ \left/\begin{array}{c}
x\mbox{ is a finitely nonzero function }\\
x(\emptyset)=0
\end{array} \right. \right\}.$$

Now consider over $N$ the norm defined in the previous section. In other words

$$\Vert x\Vert:=\sup\sum_{i=1}^k \left\vert \sum_{t\in S_i}x(t)\right\vert,$$

where the sup is taken over all families of admissible segments $\{S_1,\ldots, S_k\}$ in $T$.

Now define $M$ as the completion of $N$ under the above norm.

Note that $i:N\hookrightarrow JH_\infty$ is a linear isometry. So, it can be uniquely extended to a linear isometry $\Phi:M\longrightarrow JH_\infty$ and, as a consequence, $M$ can be view as a closed subspace of $JH_\infty$.

\begin{remark} Given $x\in N$, notice that in the definition of the norm we can consider only families of admissible segments which are $p-q$ segments with $p\geq 1$. This is an important fact which will allow us to conclude the $w^*$-strong  diameter two property for $M^*$\end{remark}

For $S\subseteq T$ a segment define $f_S\in M^*$ by

$$f_S(x)=\sum_{t\in S} x(t)\ \ \forall x\in N.$$

The first consequence of the previous Remark is that

$$A:=\left\{ \sum_{i=1}^n\lambda_i f_{S_i}\ \left/\begin{array}{c}
\vert \lambda_i\vert= 1 \ \forall i\in\{1,\ldots, n\}\\
 \{S_1,\ldots, S_n\}\mbox{ family of admissible segments}\\
 \emptyset\notin S_i\ \forall i\in\{1,\ldots,n\}\end{array} \right.\right\}$$

is a norming set in $B_{M^*}$. Hence

\begin{equation}\label{normantejhinfi}
B_{M^*}=\overline{co}^{w^*}
\left(A\right),
\end{equation}

is an immediate consequence of Hahn-Banach's theorem.

We will use that fact in order to prove that $M^*$ enjoys the $w^*$-strong diameter two property.

\begin{theorem}

$M^*$ has the $w^*$-strong diameter two property.

\end{theorem}

\begin{proof}
Let $C:=\sum_{i=1}^n \lambda_i S(B_{M^*},x_i,\varepsilon)$ a convex combination of $w^*$-slices in $B_{M^*}$, being $x_1,\ldots, x_n$ finitely non-zero functions defined on $T$. Let prove that $diam(C)=2$.

To this aim, for each $i\in\{1,\ldots, n\}$ we can find $n_i\in\mathbb N, \{S_1^i,\ldots, S_{n_i}^i\}$ a family of admissible segments in $T$ and $\mu_1^i,\ldots,\mu_{n_i}^i\in \{-1,1\}$ such that

$$\sum_{i=1}^n \lambda_i \sum_{j=1}^{n_i}\mu_j^i f_{S_j^i}\in C.$$

Now for every $i\in\{1,\ldots, n\}$ we have that $S_j^i$ is a $p_i-q_i$ segment for each $j\in\{1,\ldots, n_i\}$. We can assume that $q_1=q_2=\ldots=q_n=r$ and that $r>\max\limits_{1\leq i\leq n} p_i$  because $x_1,\ldots, x_n$ have finite support and each element on $T$ has infinitely many offspring.

Again due to the finiteness of $sup(x_i)$ for each $i\in\{1,\ldots, n\}$ we can find $B$ a branch in $T$ such that

$$B\bigcap\left(\bigcup\limits_{i=1}^n supp(x_i)\right)=\emptyset.$$

For each $i\in\{1,\ldots, n\}$ we can choose $S_i\subseteq B$ a $p_i-r$ segment in $T$. As $S_i\cap supp(x_i)=\emptyset$ and $\{S_1^i,\ldots, S_{n_i}^i, S_i\}$ is a family of admissible segments in $T$ we deduce that

$$\sum_{i=1}^n \lambda_i \left( \sum_{j=1}^{n_i}\mu_j^i f_{S_j^i}\pm f_{S_i}\right)\in C.$$

Hence

$$diam(C)\geq \left\Vert\sum_{i=1}^n \lambda_i \left( \sum_{j=1}^{n_i}\mu_j^i f_{S_j^i}+ f_{S_i}\right)-\sum_{i=1}^n \lambda_i \left( \sum_{j=1}^{n_i}\mu_j^i f_{S_j^i}- f_{S_i}\right)\right\Vert=$$

$$=2\left\Vert \sum_{i=1}^n \lambda_i f_{S_i}\right\Vert.$$

Let's prove that $\left\Vert \sum_{i=1}^n \lambda_i f_{S_i}\right\Vert=1$. Remark's that $\left\Vert \sum_{i=1}^n \lambda_i f_{S_i}\right\Vert\leq 1$ is obvious in view of triangle inequality. Moreover, as $S_i$ is a $p_i-r$ segment in $T$ and $p_i<r\ \forall i\in\{1,\ldots, n\}$ we deduce the existence of $\alpha\in \bigcap\limits_{i=1}^n S_i$. Now $e_\alpha\in S_M$. Hence

$$\left\Vert \sum_{i=1}^n \lambda_i f_{S_i}\right\Vert\geq \sum_{i=1}^n \lambda_i f_{S_i}(e_\alpha)=\sum_{i=1}^n \lambda_i=1.$$

Thus $diam(C)=2$ as desired.

\end{proof}

Last theorem shows that $M^*$ has each $w^*$ diameter two property and so the norm of $M$ is octahedral. Also, it is easy to see that $M$ has the strong diameter two property, as proved for $JH$, and so the norm of $M^*$ is octahedral. As $M$ is separable, we deduce the following

\begin{corollary} $M$ has the almost Daugavet property.\end{corollary}

\end{document}